\documentclass[reqno, oneside,11pt,letterpaper]{amsart}
\usepackage{amsfonts,amsmath,amssymb,amsthm,enumitem,url}
\newtheorem{definition}{Definition}[section]
\newtheorem{theorem}[definition]{Theorem}
\newtheorem{corollary}[definition]{Corollary}
\newtheorem{remark}[definition]{Remark}
\usepackage{geometry}[margin=1in]
\usepackage{graphicx} 

\title{Reweighting metric measure spaces and Onsager\textendash{}Machlup}
\author{Zachary Selk}
\address{Department of Mathematics\\ Florida State University Tallahassee, FL}
\email{zselk@fsu.edu}
\date{\today}

\begin{document}

\begin{abstract}
Given a metric measure space $M:=(X,d,\mu)$ the Onsager\textendash{}Machlup (OM) functional is a real valued function that has been seen as a generalized notion of a probability density function. The effect of reweighting the measure on OM functionals has been studied, however analogous reweightings of the metric to the best of our knowledge remain open. In this short note, we prove a transformation formula for OM functionals on geodesic metric measure spaces under reweighting of both the metric and the measure.
\end{abstract}

{\maketitle}
\section{Introduction}
One of the difficulties in infinite dimensional probability, which is the natural setting for stochastic processes, is the lack of a Lebesgue measure to define probability density functions. The Onsager\textendash{}Machlup (OM) functional (introduced in \cite{OM-Original}) is a generalized notion of a probability density function which often exists for general metric measure spaces, including infinite dimensional spaces. Although OM theory was introduced in the context of stochastic analysis, it is fundamentally about the interaction of measures and metrics, and is defined in terms of ratios of measures of small balls in the following way.

\begin{definition}
    Let $M:=(X,d,\mu)$ be a metric measure space and let $Z\subseteq X$. Denote by $B(r,z)$ the ball of radius $r>0$ around $z\in Z$. Suppose that 
    \begin{equation}\label{eq:raio}
        \lim_{r\to 0^+}\frac{\mu(B(r,x))}{\mu(B(r,y))}=\exp\left(\operatorname{OM}(y)-\operatorname{OM}(x)\right),
    \end{equation}
    for all $x,y\in Z$. Then $\operatorname{OM}=\operatorname{OM}(X,d,\mu)$ is called an Onsager\textendash{}Machlup (OM) functional of $M$ on $Z$.
\end{definition}
We use the indefinite article ``an" instead of the definite article ``the" because OM functionals are only defined up to additive constants. Indeed, if $\operatorname{OM}$ satisfies \eqref{eq:raio} then so does $\operatorname{OM}+C$ for any $C\in \mathbb R$. For convenience, in the sequel by ``the OM functional" we mean ``an OM functional up to an additive constant."

The OM functional of a metric measure space can be seen as a generalization of a probability density function via Lebesgue's differentiation theorem. If $M=(\mathbb R^n,d,e^{-f}\lambda)$ where $d$ is the Euclidean metric and $\lambda$ is the Lebesgue measure, then Lebesgue's differentiation theorem implies that the OM functional of $M$ is $f$. Similarly, if $M=(X,d,\mu)$ where $\mu$ is any discrete probability measure, continuity of measure implies that the OM functional of $M$ is the negative log density of the probability mass function. 

Often, especially in infinite dimensional settings, $Z$ is much smaller than $X$. Indeed, if $y\not\in \operatorname{supp}(\mu)$ then the ratio in \eqref{eq:raio} is not well defined even pre-limit. Furthermore, in the setting of infinite dimensional Gaussian measures, the OM functional is finite only on the measure's Cameron\textendash{}Martin space, which has measure $0$ (see \cite{Bogachev, Dashti} e.g.). As OM theory has been introduced in the context of probability theory, there has been much interest in transformations of OM functionals under equivalent changes of measure. For example, see \cite{OM-Gamma-1,Dashti,Self-FWOM} for various versions of the following result. 
\begin{theorem}\label{thm:fixed-metric}
    Let $M_0=(X,d,\mu_0)$ be a metric measure space with OM function $\operatorname{OM}_0$ on $Z\subseteq X$. Suppose that $\mu=e^{-V}\mu_0$ for a locally uniformly continuous function $V:X\to \mathbb R$. Then the OM function of $M=(X,d,e^{-V}\mu_0)$ on $Z$ is $\operatorname{OM}=\operatorname{OM}_0+V$. 
\end{theorem}
OM theory can be placed in the study of general metric measure spaces, so it is natural to ask what happens under changes of metric. Although it is known that even equivalent metrics can lead to different OM functionals (see Example B.4 in \cite{OM-Gamma-1}), a similar transformation formula for joint reweightings of metrics and measures (in the sense of \cite{Han-mms-conformal}) is absent from the literature. The purpose of this note is to show the following formula for the reweighting of metric measure spaces.
\begin{theorem}\label{thm:main}
Let $M:=(X,d_0,\mu_0)$ be a geodesic metric measure space with OM functional $\operatorname{OM}_0$ on $Z\subseteq X$. Let $U,V:X\to \mathbb R$ be locally uniformly continuous, and let $M:=(X,e^{-U}d_0,e^{-V}\mu_0)=:(X,d,\mu).$ Denote by $B_0(r,x)$ the $d_0$ ball of radius $r$ around $x$ and by $B(r,x)$ the $d$ ball of radius $r$ around $x$.
\begin{enumerate}[label={(\alph*)}]
\item Suppose that $U$ is constant. Then the OM functional for $M$ up to an additive constant is $$\operatorname{OM}=\operatorname{OM}_0+V.$$
\item Suppose there is a point $q\in Z$ so that 
\begin{equation}\label{eq:small-ball-C}
    \lim_{r\to 0^+}\frac{\mu_0(B_0(Cr,q))}{\mu_0(B_0(r,q))}= C^p
\end{equation}
for some $p\in \mathbb R$ and for all $C>0$. Then the OM functional for $M$ up to an additive constant is  $$\operatorname{OM}=\operatorname{OM}_0-pU+V.$$
\item Suppose there is a point $q\in Z$, some $C\in (0,\infty)$ and $\alpha>0$ so that the small ball estimate
\begin{equation}\label{eq:small-ball-est}
    \lim_{r\to 0^+}r^\alpha \log \mu_0(B_0(r,q))=-C
\end{equation}
holds. Then if $U$ is nonconstant, the OM functional is not well defined on $Z$. In particular, for every $x\in Z$ there is some $y\in Z$ so that 
$$\lim_{r\to 0^+}\frac{\mu(B(r,x))}{\mu(B(r,y))}=0$$
or 
$$\lim_{r\to 0^+}\frac{\mu(B(r,x))}{\mu(B(r,y))}=\infty.$$
\end{enumerate} 
\end{theorem}
\begin{remark}
Any $n$-dimensional Riemannian manifold $(M,g,\operatorname{vol}_g)$ satisfies \eqref{eq:small-ball-C} with $p=n$, see e.g. Theorem 3.1 in \cite{Gray-RM-small-ball}. This also implies that any finite dimensional Riemannian manifold, equipped with a measure equivalent to its volume form also satisfies \eqref{eq:small-ball-C}, such as a smooth metric measure space. Small ball estimates of the form \eqref{eq:small-ball-est} are ubiquitous in infinite dimensional probability. It would be impossible to name all examples but a very small sampling include solutions to SPDEs \cite{Davar-Mueller-small-ball}, (fractional) Brownian motion and more general Gaussian processes \cite{Small-ball-Gaussian}, Brownian sheets \cite{Brownian-sheet-small-ball}, and integrated Brownian motion \cite{Integrated-BM-small-ball}.  Parts (b) and (c) of Theorem \ref{thm:main} can, therefore, be seen as covering many examples of finite and infinite dimensional spaces, respectively. 
\end{remark}
We conclude the introduction by mentioning the definition of $e^{-U}d_0$. As we assume that $M_0$ is geodesic, we have that
\begin{equation}
    d_0(x,y)=\inf\left\{\int_0^1 |\dot \gamma(t)|dt: \gamma\in \operatorname{AC}([0,1],X),\gamma(0)=x,\gamma(1)=y\right\},
\end{equation}
where $\operatorname{AC}([0,1],X)$ is the set of absolute continuous paths into $X$. We therefore define $d:=e^{-U}d_0$ by
\begin{equation}\label{def:d}
    d(x,y)=\inf\left\{\int_0^1 e^{-U(\gamma(t))}|\dot \gamma(t)|dt: \gamma\in \operatorname{AC}([0,1],X),\gamma(0)=x,\gamma(1)=y\right\}.
\end{equation}
As $U$ is locally uniformly continuous, $U$ is locally bounded. Therefore $d$ is topologically equivalent to $d_0$. If $U$ is also globally bounded, then $d$ is bi-Lipschitzly equivalent to $d_0$ in the sense that there is $C>0$ so that $C^{-1}d_0(x,y)\leq d(x,y)\leq Cd_0(x,y)$ for all $x,y\in X$. 
\section{Proof of main theorem}
\begin{proof}
\begin{enumerate}[label={(\alph*)}]
    \item If $U=C$ is constant, then $$\frac{\mu(B(r,x))}{\mu(B(r,y))}=\frac{\mu(B_0(re^C,x))}{\mu(B_0(re^C,y))}$$
    and Theorem \ref{thm:fixed-metric} concludes.
    \item First, we show that, if \eqref{eq:small-ball-C} holds, then it holds for all $z\in Z$. To that end, let $z\in Z$. Then 
\begin{align*}
   \lim_{r\to 0^+} \frac{\mu_0(B_0(Cr,z))}{\mu_0(B_0(r,z))}&= \lim_{r\to 0^+}\frac{\mu_0(B_0(Cr,z))}{\mu_0(B_0(Cr,q))}\frac{\mu_0(B_0(r,q))}{\mu_0(B_0(r,z))}\frac{\mu_0(B_0(Cr,q))}{\mu_0(B_0(r,q))}\\
   &=\exp\left(\operatorname{OM}_0(q)-\operatorname{OM}_0(z)-\operatorname{OM}_0(q)+\operatorname{OM}_0(q)\right)C^p\\
    &=C^p.
\end{align*}
Now we show that \eqref{eq:small-ball-C} holds for $\mu$. As $V$ is locally uniformly continuous, for any $r>0$ and $x\in X$, there exists a continuous, increasing function $\beta_{r,x}:[0,\infty)\to [0,\infty)$ with $\lim_{s\to 0^+}\beta_{r,x}(s)=0$ so that, for any $y\in B_0(r,x)$, we have $$|V(y)-V(x)|\leq \beta_{r,x}(d_0(x,y))\leq \beta_{r,x}(r).$$
Therefore we have that
$$ \frac{\mu(B_0(Cr,x))}{\mu(B_0(r,x))}=\frac{\int_{B_0(Cr,x)}e^{-V(y)}d\mu_0(y)}{\int_{B_0(r,x)}e^{-V(y)}d\mu_0(y)}\leq \frac{\mu_0(B_0(Cr,x))}{\mu_0(B_0(r,x))}\frac{e^{-V(x)+\beta_{x,Cr}(r)}}{e^{-V(x)-\beta_{x,r}(r)}}
$$
and similarly
$$\frac{\mu_0(B_0(Cr,x))}{\mu_0(B_0(r,x))}\frac{e^{-V(x)-\beta_{x,Cr}(r)}}{e^{-V(x)+\beta_{x,r}(r)}}\leq \frac{\mu(B_0(Cr,x))}{\mu(B_0(r,x))}.$$
Letting $r\to 0^+$ implies that
$$\lim_{r\to 0^+}\frac{\mu(B_0(Cr,x))}{\mu(B_0(r,x))}=C^p.$$

As $U$ is locally uniformly continuous, for any $r>0$ and $x\in X$, there exists a continuous, increasing function $\omega=\omega_{r,x}:[0,\infty)\to [0,\infty)$ with $\lim_{s\to 0^+}\omega(s)=0$ so that for any $y\in B_0(r,x)$ we have $$|U(y)-U(x)|\leq \omega(d_0(x,y))\leq \omega(r).$$
Therefore for $y\in B_0(r,x)$ we have from the definition of $d$ given in \eqref{def:d} that $$d_0(x,y)e^{-U(x)-\omega(r)}\leq d(x,y)\leq d_0(x,y) e^{-U(x)+\omega(r)}.$$
This implies the inclusions
$$B_0(re^{U(x)-\omega(r)},x)\subseteq B(r,x)\subseteq B_0(re^{U(x)+\omega(r)},x).$$
Furthermore, for any $R>r$ we also have that 
$$B_0(re^{U(x)-\omega(R)},x)\subseteq B(r,x)\subseteq B_0(re^{U(x)+\omega(R)},x).$$
Therefore for $x,y\in Z$, setting $\omega(s):=\max(\omega_{r,x}(s),\omega_{r,y}(s))$ gives
\begin{align*}
    \lim_{r\to 0^+}\frac{\mu(B(r,x))}{\mu(B(r,y))}&\leq  \lim_{r\to 0^+}\frac{\mu(B_0(re^{U(x)+\omega(R)},x))}{\mu(B_0(re^{U(y)-\omega(R)},y))} \\
    &= \lim_{r\to 0^+}\frac{\mu(B_0(re^{U(x)+\omega(R)},x))}{\mu(B_0(r,x))}\frac{\mu(B_0(r,y))}{\mu(B_0(re^{U(y)-\omega(R)},y))}\frac{\mu(B_0(r,x))}{\mu(B_0(r,y))}\\
    &=e^{p(U(x)+\omega(R))}e^{-p(U(y)-\omega(R))}e^{\operatorname{OM}_0(y)+U(y)-\operatorname{OM}(x)-U(y)},
\end{align*}
where we used Theorem \ref{thm:fixed-metric}. Similarly, we have that 
$$\lim_{r\to 0^+}\frac{\mu(B(r,x))}{\mu(B(r,y))}\geq e^{p(U(x)-\omega(R))}e^{-p(U(y)+\omega(R))}e^{\operatorname{OM}_0(y)+U(y)-\operatorname{OM}(x)-U(y)}.$$
Sending $R\to 0^+$ concludes. 
\item First, we show that, if \eqref{eq:small-ball-est} holds, then it holds for all $z\in Z$. To that end, let $z\in Z$. Then 
\begin{align*}
    \lim_{r\to 0^+}r^\alpha \log \mu_0(B_0(r,z))&=  \lim_{r\to 0^+}r^\alpha \log \left(\mu_0(B_0(r,q))\frac{\mu_0(B_0(r,z))}{\mu_0(B_0(r,q))}\right)\\
    &=\lim_{r\to 0^+}r^\alpha \log \left(\mu_0(B_0(r,q))\right)+\lim_{r\to 0^+}r^\alpha \log \left(\frac{\mu_0(B_0(r,z))}{\mu_0(B_0(r,q))}\right)\\
    &=-C+0\times \left(\operatorname{OM}_0(q)-\operatorname{OM}_0(z)\right)\\
    &=-C.
\end{align*}
Now we show that \eqref{eq:small-ball-est} holds for $\mu$. As $V$ is locally uniformly continuous, for any $r>0$ there exists a continuous, increasing function $\beta_{r,q}:[0,\infty)\to [0,\infty)$ with $\lim_{s\to 0^+}\beta_{r,q}(s)=0$ so that for any $y\in B_0(r,q)$ we have $$|V(y)-V(q)|\leq \beta_{r,q}(d_0(q,y))\leq \beta_{r,q}(r).$$
Therefore, we have that
\begin{align*}
    \lim_{r\to 0^+}r^\alpha \log \mu(B_0(r,q))&= \lim_{r\to 0^+}r^\alpha \log\int_{B_0(r,q)}e^{-V(y)}d\mu_0(y)\\
    &\leq \lim_{r\to 0^+}r^\alpha \log \left(e^{\beta_{r,q}(r)+V(q)}\mu_0(B_0(r,q))\right)\\
    &=-C,
\end{align*}
and the lower bound follows similarly. 

Now we are able to prove (c). For any $x\in Z$, as $U$ is nonconstant without loss of generality there is some $y$ with $U(x)>U(y)$. Then similarly to the proof of part (b), we have for $r>0$ that
\begin{align*}
    \frac{\mu(B(r,x))}{\mu(B(r,y))}&\leq  \frac{\mu(B_0(re^{U(x)+\omega(R)},x))}{\mu(B_0(re^{U(y)-\omega(R)},y))} \\
    &= \frac{\mu(B_0(re^{U(x)+\omega(R)},x))}{\mu(B_0(r,x))}\frac{\mu(B_0(r,y))}{\mu(B_0(re^{U(y)-\omega(R)},y))}\frac{\mu(B_0(r,x))}{\mu(B_0(r,y))}\\
    &=:B.
\end{align*}
The small ball estimate \eqref{eq:small-ball-est} and Theorem \ref{thm:fixed-metric} imply that $B$ is asymptotic to 
\begin{align*}
    B&\sim \frac{e^{-C/(re^{U(x)+\omega(R)})^\alpha}}{e^{-C/r^\alpha}}\frac{e^{-C/r^\alpha}}{e^{-C/(re^{U(y)-\omega(R)})^\alpha}}e^{\operatorname{OM}_0(y)+U(y)-\operatorname{OM}_0(x)-U(x)}\\
    &=\frac{e^{-C/(re^{U(x)+\omega(R)})^\alpha}}{e^{-C/(re^{U(y)-\omega(R)})^\alpha}}e^{\operatorname{OM}_0(y)+U(y)-\operatorname{OM}_0(x)-U(x)}\\
    &=\exp\left(\frac{C}{r^\alpha}\left(\frac{1}{(e^{U(y)-\omega(R)})^\alpha}-\frac{1}{(e^{U(x)+\omega(R)})^\alpha}\right)\right)e^{\operatorname{OM}_0(y)+U(y)-\operatorname{OM}_0(x)-U(x)}.
\end{align*}
As $U(x)>U(y)$ we can choose $r, R$ small enough so that
$$\frac{1}{(e^{U(y)-\omega(R)})^\alpha}-\frac{1}{(e^{U(x)+\omega(R)})^\alpha}<0.$$
Therefore 
$$\lim_{r\to 0^+}\frac{\mu(B(r,x))}{\mu(B(r,y))}=0,$$
concluding the proof of Theorem \ref{thm:main}.
\end{enumerate}

\end{proof} 
\section{Corollaries and discussion}
There are two immediate corollaries. 
\begin{corollary}
    Let $M_0=(\mathbb R^n, d_0,e^{-f}\lambda)$, where $d_0$ is the Euclidean distance and $\lambda$ is Lebesgue measure. Let $(\mathcal B, \|\cdot\|)$ be any normed vector space of functions $\mathbb R^n\to \mathbb R$ so that locally uniformly continuous functions are dense in $\mathcal B$. Then for any function $h\in \mathcal B$ and any $\varepsilon>0$, there is a choice of metric $d_{h,\varepsilon}$ so that the OM functional of $M_{h,\varepsilon}=(\mathbb R^n, d_{h,\varepsilon},e^{-f}\lambda)$ satisfies 
    \begin{equation}
        \|\operatorname{OM}(\mathbb R^n, d_{h,\varepsilon},e^{-f}\lambda)-h\|<\varepsilon.
    \end{equation}
\end{corollary}
\begin{corollary}
    Let $M_0=(X,d_0,\mu_0)$ satisfy the conditions of Theorem \ref{thm:main}, part (c). Then if $d$ is another metric conformally equivalent to $d_0$, with a locally uniformly continuous conformal density, so that the OM functional of $M=(X,d,\mu_0)$, $\operatorname{OM}(X,d,\mu_0)$ is well defined on $Z$, then $$\operatorname{OM}(X,d,\mu_0)=\operatorname{OM}(X,d_0,\mu_0).$$
\end{corollary}
OM functionals in many finite dimensional settings can be freely made to be any measurable function but in many infinite dimensional settings conformally changing the metric destroys the existence of the OM functional. Therefore from the standpoint of OM theory, reweighting measures is equivalent to reweighting metrics in finite dimensions. However they are non-equivalent in infinite dimensions. Also arguably, these corollaries show that OM functionals of a fixed probability measure are more ``canonical" in infinite dimensions than they are in finite dimensions. 

One further corollary is the following, saying that we can ``uniformize" many measures on $\mathbb R^d$ by reweighting the geometry.
\begin{corollary}\label{cor:cart}
    Let $M_0=(\mathbb R^d,d_0, e^{-f}\lambda)$, where $d_0$ is the Euclidean distance, $\lambda$ is the Lebesgue measure and $f$ is continuous. Then there is a choice of ``uniformization" metric, $d_{\text{uniform}}$ so that the OM functional of $M=(\mathbb R^d,d_{\text{uniform}},e^{-f}\lambda)$ satisfies 
    $$\operatorname{OM}(x)=0$$
    for all $x\in \mathbb R^d$.
\end{corollary}
\begin{proof}
    Any continuous function on $\mathbb R^d$ is locally uniformly continuous, so let $d_{\text{uniform}}=e^{f/d}d_0$.
\end{proof}
The uniformization metric can be seen analogously to a popular data visualization technique called a cartogram, where a map is rescaled so data is displayed geometrically. By distorting the geometry of a map, we can make a particular distribution (such as population density on Earth, see Figure \ref{fig:cart}) uniform. Our Theorem \ref{thm:main}, part (c), may be understood as the statement that cartograms are impossible in many infinite dimensional settings. 

\bibliographystyle{plain}
\bibliography{References}

\begin{figure}
    \hspace{-3cm}\includegraphics[width=19cm, angle=90 ]{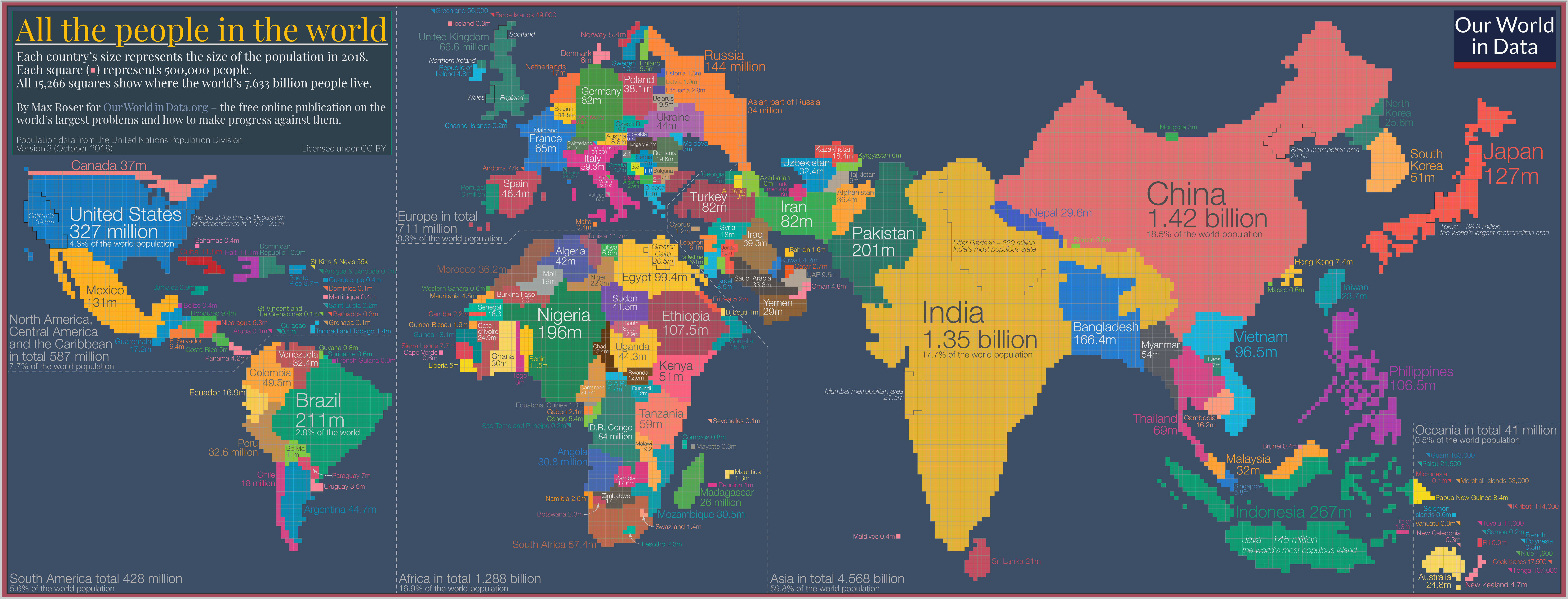}
    \caption{A cartogram of Earth, reweighted by population, taken from \cite{Cartogram}. By distorting the geometry of Earth, we can make population density uniform. This can be seen as analogous to Corollary \ref{cor:cart}.}
    \label{fig:cart}
\end{figure}

\end{document}